\documentclass[12pt,reqno]{amsart}

\usepackage{amssymb}
\usepackage{amsmath}
\usepackage{latexsym}
\usepackage{amsthm}
\usepackage{epsfig}
\usepackage{color}
\usepackage{amsfonts,amssymb,mathrsfs,amscd}

\theoremstyle{plain}
\newtheorem{theorem}{Theorem}[section]
\newtheorem{proposition}{Proposition}[section]
\newtheorem{corollary}{Corollary}[section]

\newtheorem{remark}{\bf Remark}[section]
\theoremstyle{definition}

\renewcommand{\div}{\mathop\mathrm{div}}
\newcommand{\Tr}{\mathop\mathrm{Tr}}

\title
[Lieb--Thirring inequalities on the torus ]
 {Lieb--Thirring inequalities on the torus}

\author[A.Ilyin and A.Laptev] {Alexei  Ilyin and Ari Laptev}

\begin{document}
 \begin{abstract}
We consider  the Lieb--Thirring inequalities on the
$d$-dimen\-sional torus with arbitrary periods. In the space of
functions with zero average with respect to the shortest coordinate
we prove the Lieb--Thirring inequalities for the $\gamma$-moments
of the negative eigenvalues with constants independent of ratio of
the periods. Applications to the attractors of the damped
Navier--Stokes system are given.
\end{abstract}


\address
{\noindent\newline  Keldysh Institute of Applied Mathematics;
\newline
Imperial College London and Institute Mittag--Leffler;}
\email{ ilyin@keldysh.ru;  a.laptev@imperial.ac.uk}

\maketitle

\setcounter{equation}{0}
\section{Introduction}
\label{sec0}

The Lieb--Thirring inequalities~\cite{LT} give estimates for the
$\gamma$-moments of the negative eigenvalues $-\nu_j$  of the
Schr\"odinger operator
\begin{equation}\label{Schr}
-\Delta-V
\end{equation}
in $L_2(\mathbb{R}^d)$:
\begin{equation}\label{est-L-T}
\sum_{\nu_i\le0}|\nu_i|^\gamma\le\mathrm{L}_{\gamma,d}
 \int V(x)^{\gamma+ d/2}dx.
\end{equation}
Here $V\ge0$ is a real valued potential sufficiently fast decaying at  infinity.

Sharp values of the constants $\mathrm{L}_{\gamma,d}$ for all $\gamma\ge3/2$ and all $d$
were found in~\cite{Lap-Weid}:
\begin{equation*}\label{LW32}
\mathrm{L}_{\gamma,d}=\mathrm{L}_{\gamma,d}^{\mathrm{cl}},
\end{equation*}
where
\begin{equation}\label{class}
\mathrm{L}_{\gamma,d}^{\mathrm{cl}} = \frac{1}{(2\pi)^d}\, \int_{\mathbb R^d}  (1-|\xi|^2)_+^\gamma\, d\xi=
\frac{\Gamma(\gamma+1)}{2^d\pi^{d/2}\Gamma(\gamma+d/2+1)}
\end{equation}
is the semi-classical constant.
For $1\le\gamma<3/2$ the best known bounds for $\mathrm{L}_{\gamma,d}$
 were found in~\cite{D-L-L}:
\begin{equation}\label{DLL}
\mathrm{L}_{\gamma,d}\le\frac\pi{\sqrt{3}}\mathrm{L}_{\gamma,d}^{\mathrm{cl}}.
\end{equation}

 Both in~\cite{Lap-Weid} and~\cite{D-L-L} (see also~\cite{B-L} and \cite{H-L-W})
the key role is played by  the corresponding one-dimensional Lieb--Thirring
estimates for matrix-valued potentials.

If we consider the Lieb--Thirring inequalities on a compact
manifold $M$, then we have to take care on the simple eigenvalue
$0$ of the Laplacian, so that instead of~\eqref{Schr} we consider the
Schr\"odinger operator
\begin{equation*}\label{SchrM}
-\Delta-\Pi(V\Pi\cdot),
\end{equation*}
where $\Pi$ is the orthoprojection
\begin{equation*}\label{Pi}
\Pi\varphi(x)=\varphi(x)-\frac1{|M|}\int_{M}\varphi(x)dM,
\end{equation*}
and $|M|$ is the measure of $M$. Then estimate~\eqref{est-L-T} still
holds on $M$ with constant $\mathrm{L}_{\gamma,d}=\mathrm{L}_{\gamma,d}(M)$
depending on the geometric properties of $M$.

The spectral inequality~(\ref{est-L-T}) for $1$-moments
 is equivalent to the inequality for orthonormal families.
Let $\{\varphi_j\}_{j=1}^N\in H^1(\mathbb{R}^d)$ be an
orthonormal family in $L_2(\mathbb{R}^d)$.
(In the case of a manifold we have to assume that $\int_M\varphi_j(x)dM=0$.)
Then
$\rho_\varphi(x):=\sum_{j=1}^N\varphi_j(x)^2$ satisfies the
inequality
\begin{equation}\label{L-T-orth-orig}
\int\rho_\varphi(x)^{1+2/d}dx\le
\mathrm{k}_{d}\sum_{j=1}^N\|\nabla\varphi_j\|^2,
\end{equation}
where the best constants $\mathrm{k}_{d}$ and  $\mathrm{L}_{1,d}$
satisfy \cite{LT}, \cite{Lieb}
\begin{equation}\label{k=L}
\mathrm{k}_{d}=(2/d)(1+d/2)^{1+2/d}\mathrm{L}_{1,d}^{2/d}.
\end{equation}

In addition to the initial quantum mechanical applications
inequality~(\ref{L-T-orth-orig}) is essential for finding good
estimates for the dimension of the attractors  in the theory of
infinite dimensional dissipative dynamical systems, especially, for
the attractors of the Navier--Stokes equations (see, for instance,
\cite{Lieb}, \cite{B-V}, \cite{CF88},\cite{G-M-T}, \cite{I-MS2005},
\cite{T} and the references therein). Lieb-Thirring
inequalities~(\ref{L-T-orth-orig}) were generalized to higher-order
elliptic operators on domains with various boundary conditions and
Riemannian manifolds \cite{G-M-T}, \cite{T}, however, with no
information  the corresponding constants. A different approach
based on the methods of trigonometric series was proposed
in~\cite{Kashin}.

For the two-dimensional square torus $\mathbb{T}^2$ an explicit
bound for the Lieb--Thirring constant
$\mathrm{L}_{1,2}(\mathbb{T}^2)$ was obtained in~\cite{I12JST}:
\begin{equation}\label{T2}
\mathrm{L}_{1,2}(\mathbb{T}^2)\le\frac38\,.
\end{equation}
Following the original approach in~\cite{LT} the Birman--Schwinger
kernel was used in~\cite{I12JST}, and the bound $3/8$ is the same
as that for $\mathbb{R}^2$ in~\cite{LT}.

The Lieb--Thirring constants on the torus depend only on the ratio of the periods
and on the elongated torus $\mathbb{T}^2_\alpha=(0,2\pi/\alpha)\times(0,2\pi)$
are unbounded as $\alpha\to0$. This is due to the fact that the space $L_2(\mathbb{T}^2_\alpha)$
contains the subspace of (periodic) functions depending only on the long coordinate
$x_1$, which gives that, for example, for the $1$-moments we have
$$
\mathrm{L}_{1,2}(\mathbb{T}^2_\alpha)=\frac{\mathrm{L}_{1,2}(\mathbb{T}^2)}{\alpha}\,.
$$
To see this we extend the functions in the direction $x_2$ by periodicity
$1/\alpha$ times (assuming that $1/\alpha$ is integer) to reduce the treatment to the
square torus, see~\cite{G-T}.

The orthogonal complement to the subspace $L_2(0,2\pi/\alpha)$ of functions
depending only on $x_1$ consists of functions $\varphi(x_1,x_2)$ that
have zero average with respect to $x_2$ for all $x_1$:
$$
\biggl\{\varphi(x),\ \int_0^{2\pi}\varphi(x_1,t)dt=0,\ \forall x_1\in(0,2\pi/\alpha)\biggr\}.
$$
Let $\mathrm{P}$ denote the corresponding orthoprojection.  Then the Lieb--Thirring
constants for the operator
\begin{equation}\label{SchrP}
\mathcal{H}=-\Delta-\mathrm{P}(V(x)\mathrm{P}\,\cdot)
\end{equation}
on $\mathbb{T}_\alpha^2$ are independent of $\alpha$ (more precisely, have $\alpha$-independent upper bounds). This was first observed in~\cite{Z}, and explicit estimates were obtained in~\cite{I-MS2005}.

In this work we consider Lieb--Thirring inequalities for the operator \eqref{SchrP} on the $d$-dimensional torus $\mathbb{T}^d_\alpha$
\begin{equation}\label{T2a}
\mathbb{T}^d_\alpha=(0,L_1)\times\dots\times(0,L_{d-1})
\times(0,2\pi),
\end{equation}
where $L_i=2\pi/\alpha_i$ and  the lengths of the periods are arranged in the non-increasing order
\begin{equation}\label{alpha}
\alpha_1\le\dots\le\alpha_{d-1}\le\alpha_d=1.
\end{equation}
Here $\mathrm{P}$ is the orthoprojection
\begin{equation}\label{projP1}
(\mathrm{P}\psi)(x_1,\dots,x_d)=\psi(x_1,\dots,x_d)-\frac1{2\pi}\int_0^{2\pi}
\psi(x_1,\dots,x_{d-1},t)dt,
\end{equation}
so that the resulting function  has zero average  with respect to shortest
coordinate  $x_d$ for all $x_1,\dots,x_{d-1}$.

We can now state our main result.

\begin{theorem}\label{T:intro} Let $d\le19$. Then for any $\alpha$
satisfying~\eqref{alpha}
the negative eigenvalues of the  operator~\eqref{SchrP} on the torus $\mathbb{T}^d_\alpha$
satisfy for $\gamma\ge1$  the bound~\eqref{est-L-T} with
$$
\mathrm{L}_{\gamma,d}\le\left(\frac\pi{\sqrt{3}}\right)^d\,
\mathrm{L}^\mathrm{cl}_{\gamma,d},
$$
where $\mathrm{L}^\mathrm{cl}_{\gamma,d}$ is the semi-classical constant~\eqref{class}.
\end{theorem}

The main idea is to write the Laplacian in~\eqref{SchrP} in the form
$$
-\Delta=\left(-\frac{d^2}{dx_1^2}+\alpha_1^2\beta_1\right)+\dots+
\left(-\frac{d^2}{dx_{d-1}^2}+\alpha_{d-1}^2\beta_{d-1}\right)+
\left(-\frac{d^2}{dx_d^2}-\delta\right),
$$
where
$$
\delta=\alpha_1^2\beta_1+\dots+\alpha_{d-1}^2\beta_{d-1},
$$
and where
$\beta_j>0$ are chosen so that $\delta<1$.

For $j=1,\dots,d-1$ each operator
$$
-\frac{d^2}{dx_{j}^2}+\alpha^2\beta_j,\qquad j=1,\dots,d-1
$$
is invertible on $L_2(0,2\pi/\alpha_j)_{\mathrm{per}}$, while for
$\delta<1$ the operator
$$
-\frac{d^2}{dx_d^2}-\delta
$$
is invertible on $\dot L_2(0,2\pi)=L_2(0,2\pi)\cap\{\psi:\
\int_0^{2\pi}\psi(t)dt=0\}$.

Accordingly, in Section~\ref{sec1}, we consider for these two types of operators the
interpolation inequalities of $L_\infty\!-\!L_2\!-\!L_2$ type:\begin{equation}\label{ineqadd}
\|u\|_\infty^2\le K_1(\beta)
\left(\int_0^{2\pi/\alpha}\bigl(u'(x)^2+\alpha^2\beta u(x)^2\bigr)dx\right)^{1/2}\!
\left(\int_0^{2\pi/\alpha}u(x)^2dx\right)^{1/2},
\end{equation}
where $\beta>0$ and $u\in H^1(0,2\pi/\alpha)_\mathrm{per}$; and
\begin{equation}\label{ineqsubtract}
\|u\|_\infty^2\le K_2(\beta)
\left(\int_0^{2\pi}\bigl(u'(x)^2-\beta u(x)^2\bigr)dx\right)^{1/2}\!
\left(\int_0^{2\pi}u(x)^2dx\right)^{1/2},
\end{equation}
where $\beta<1$ and   $u\in \dot H^1(0,2\pi/\alpha)_\mathrm{per}=H^1(0,2\pi/\alpha)\cap\int_0^{2\pi}u(x)dx=0$.
We find sharp constants in these inequalities and, in particular, show that
$K_1(\beta)=1$ for $\beta\ge\beta_*=0.045\dots$ and
$K_2(\beta)=1$ for $\beta\le\beta_{**}=0.839\dots$.

In Section~\ref{sec2} we use the information on the sharp constants
$K_1(\beta)$ and $K_2(\beta)$ to obtain following \cite[Theorem
6.1]{IMRN} one dimensional inequalities of the type
\eqref{L-T-orth-orig} for traces of matrices built from
orthonormal families of \emph{vector} functions. Then using the
equivalence
 \eqref{k=L} we recast these results into the Lieb--Thirring inequalities
 for the negative eigenvalues $\{-\lambda_j\}$ and $\{-\mu_j\}$ of the one dimensional operators with \emph{matrix}-valued potentials
\begin{align}
H_1(\beta)&=-\frac {d^2}{dx^2}+\alpha^2\beta-V(x),\quad\beta>0,
\quad x\in(0,2\pi/\alpha),\label{H1}\\
H_2(\delta)&=-\frac {d^2}{dx^2}-\delta-\mathrm{P}(V(x)\mathrm{P}\cdot),\quad \delta\in[0,1),
\quad x\in(0,2\pi),\label{H2}
\end{align}
acting on $L_2(0,2\pi/\alpha)$ and $\dot L_2(0,2\pi)=\{f\in L_2(0,2\pi),\
\int_0^{2\pi}f(x)dx=0\}$, respectively.

For a non-negative $M\times M$ Hermitian matrix $V$ we obtain
\begin{align}
&\sum_j\lambda_j\le\frac 2{3\sqrt{3}}K_1(\beta)
\int_0^{2\pi/\alpha}  \Tr[V(x)^{3/2}]dx,\label{neg1}\\
&\sum_j\mu_j\le\frac 2{3\sqrt{3}}K_2(\delta)
\int_0^{2\pi}  \Tr[V(x)^{3/2}]dx.\label{neg2}
\end{align}
Then we use  the Aizenmann-Lieb argument \cite{AisLieb}
to obtain the  estimates for the Riesz means of order $\gamma\ge1$.

In Section~\ref{sec3} we use the lifting argument with respect to
dimensions~\cite{Lap-Weid} and the one-dimensional Lieb--Thirring
inequalities for matrix-valued potentials from Section~\ref{sec2} to
prove Theorem~\ref{T:ddim} (to which Theorem~\ref{T:intro} is a corollary).
 We point out here that the condition $d\le19$ (which is $d\le [\beta_{**}/\beta_*]+1$)
makes it possible to chose $\beta_j\ge\beta_*$ and
$\delta\le\beta_{**}$ so that in \eqref{neg1}, \eqref{neg2} we have
$K_1(\beta_j)=1$ and $K_2(\delta)=1$.

Finally, in Section~\ref{sec3} we give applications to
the attractors of  the damped-driven Navier--Stokes equations
on the square and elongated torus and improve the estimates of the
dimension of the attractor obtained earlier in~\cite{I-M-T} and \cite{I-T}.

\setcounter{equation}{0}
\section{Two interpolation inequalities}
\label{sec1}
\subsection*{First inequality.}
Let $x\in[0,2\pi/\alpha]_{\mathrm{per}}$ and let $\alpha>0$. We consider the
 interpolation inequality~\eqref{ineqadd}
where $u\in H^1(0,2\pi/\alpha)_{\mathrm{per}}$,  $\beta>0$
 (and no orthogonality to constants is assumed).
 More precisely, we are interested in the value of the sharp
constant $K_1(\beta)$ in this inequality.

The system
$$
\left\{\sqrt{\frac\alpha{2\pi}}\,e^{ik\alpha x}\right\}_{k\in\mathbb{Z}}
$$
is a complete orthonormal system of eigenfunctions of the operator
$-\frac{d^2}{dx^2}$ with periodic boundary conditions. Therefore the
Green's function $G_\lambda(x,\xi)$, that is, the solution
of the equation
$$
\mathbb{A}(\lambda)G_\lambda(x,\xi)=\delta(x-\xi),
$$
where
$$
\mathbb{A}(\lambda)=-\frac{d^2}{dx^2}+\alpha^2\beta+\lambda,\qquad\lambda>0,
$$
is given by the series
$$
G_\lambda(x,\xi)=\frac\alpha{2\pi}\sum_{k\in\mathbb{Z}}
\frac{e^{ik\alpha(x-\xi)}}{\alpha^2k^2+\alpha^2\beta+\lambda}\,.
$$
On the diagonal
\begin{equation}\label{G1}
G_\lambda(\xi,\xi)=\frac\alpha{2\pi}\sum_{k\in\mathbb{Z}}
\frac1{\alpha^2k^2+\alpha^2\beta+\lambda}=:g_\beta(\lambda).
\end{equation}
Here we omit the subscript $\alpha$ on the right because $\alpha$ is fixed and,
secondly, the sharp constant $K_1(\beta)$ is independent of
$\alpha$, as our final result shows.
Using the general result (Theorem~2.2 in~\cite{IZ} with $\theta=1/2$)
we have the following expression for the sharp constant
$K_1(\beta)$:
\begin{equation}\label{Kbeta}
\aligned
K_1(\beta)=
\frac{1}{\theta^\theta(1-\theta)^{1-\theta}}\cdot
\sup_{\lambda>0}\lambda^{\theta}g_\beta(\lambda)=
2\sup_{\lambda>0}\sqrt{\lambda}g_\beta(\lambda)
\endaligned
\end{equation}
(We point out that for the proof of \eqref{Kbeta} we can use the direct argument similar to the one used in  Remark~3.1 in \cite{IMRN}.)
 Next, using the formula
\begin{equation}\label{formsum}
\sum_{k\in\mathbb{Z}}\frac1{k^2+\lambda}=
\frac{\pi\coth(\pi\sqrt\lambda\,)}{\sqrt{\lambda}}
\end{equation}
in \eqref{G1} we find that
$$
g_\beta(\lambda)=\frac1{2\alpha}\frac{\coth(\pi\sqrt{\beta+\lambda/\alpha^2})}
{\sqrt{\beta+\lambda/\alpha^2}}\,.
$$
Since $\alpha>0$ is fixed, we can replace the variable $\lambda$ in
\eqref{Kbeta} by $\alpha^2\lambda$, which finally gives
\begin{equation}\label{Kbeta1}
K_1(\beta)=
\sup_{\lambda>0}\sqrt{\lambda}\,
\frac{\coth(\pi\sqrt{\beta+\lambda})}
{\sqrt{\beta+\lambda}}\,.
\end{equation}
Since for every fixed $\beta>0$
$$
\lim_{\lambda\to\infty}\sqrt{\lambda}\,
\frac{\coth(\pi\sqrt{\beta+\lambda})}
{\sqrt{\beta+\lambda}}=1,
$$
it follows that
$$
K_1(\beta)\ge1.
$$
Next, \eqref{G1} and \eqref{Kbeta}  show that
$K_1(\beta)$ is monotone non-increasing
$$
\beta_1<\beta_2\quad\Rightarrow\quad g_{\beta_1}(\lambda)>g_{\beta_2}(\lambda)
\quad\Rightarrow\quad K_1(\beta_1)\ge K_1(\beta_2),
$$
and, finally,
$$
K_1(\beta_0)=1 \quad\Rightarrow\quad K_1(\beta)=1\quad \text{for}\quad \beta\ge\beta_0.
$$

The graphs of the function $\lambda\to\sqrt{\lambda}\,
\frac{\coth(\pi\sqrt{\beta+\lambda})}
{\sqrt{\beta+\lambda}}$ are shown in Fig.~\ref{fig:K1}.
For very small $\beta$ the graphs have a sharp peak near
the origin. For  $\beta_*=0.045\dots$ (found numerically)
the supremum is equal to $1$. Hence $K_1(\beta)=1$
for all $\beta\ge\beta_*$.

Thus, we have proved the following theorem.
\begin{theorem}\label{Th:K1}
The sharp constant $K_1(\beta)$ in inequality~\eqref{ineqadd}
is given by \eqref{Kbeta1}. Furthermore,
for all $\beta\ge\beta_*=0.045\dots$, $K_1(\beta)=1$.
\end{theorem}

\begin{figure}[htb]
\centerline{\psfig{file=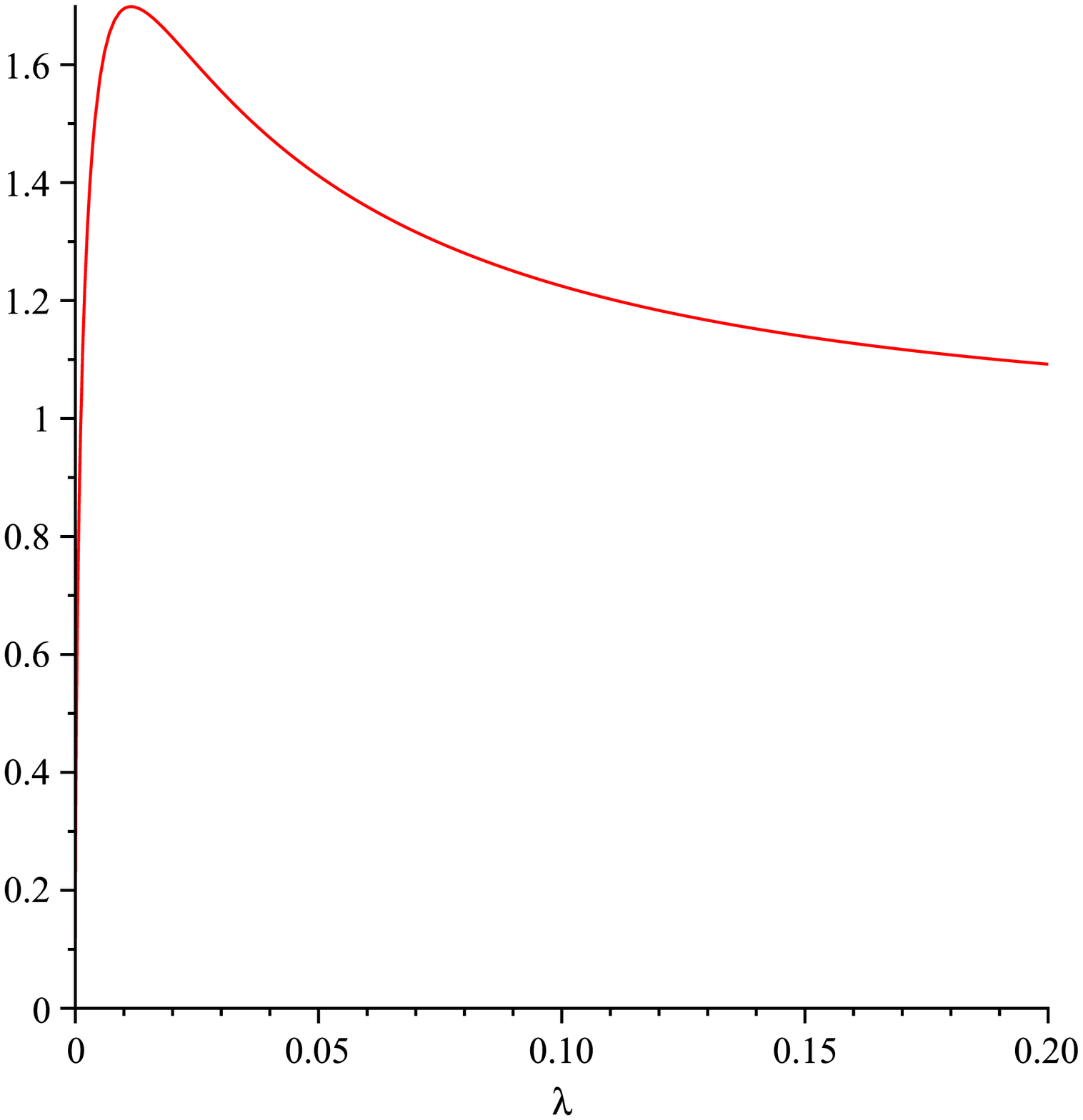,width=4cm,height=5cm,angle=0}
\psfig{file=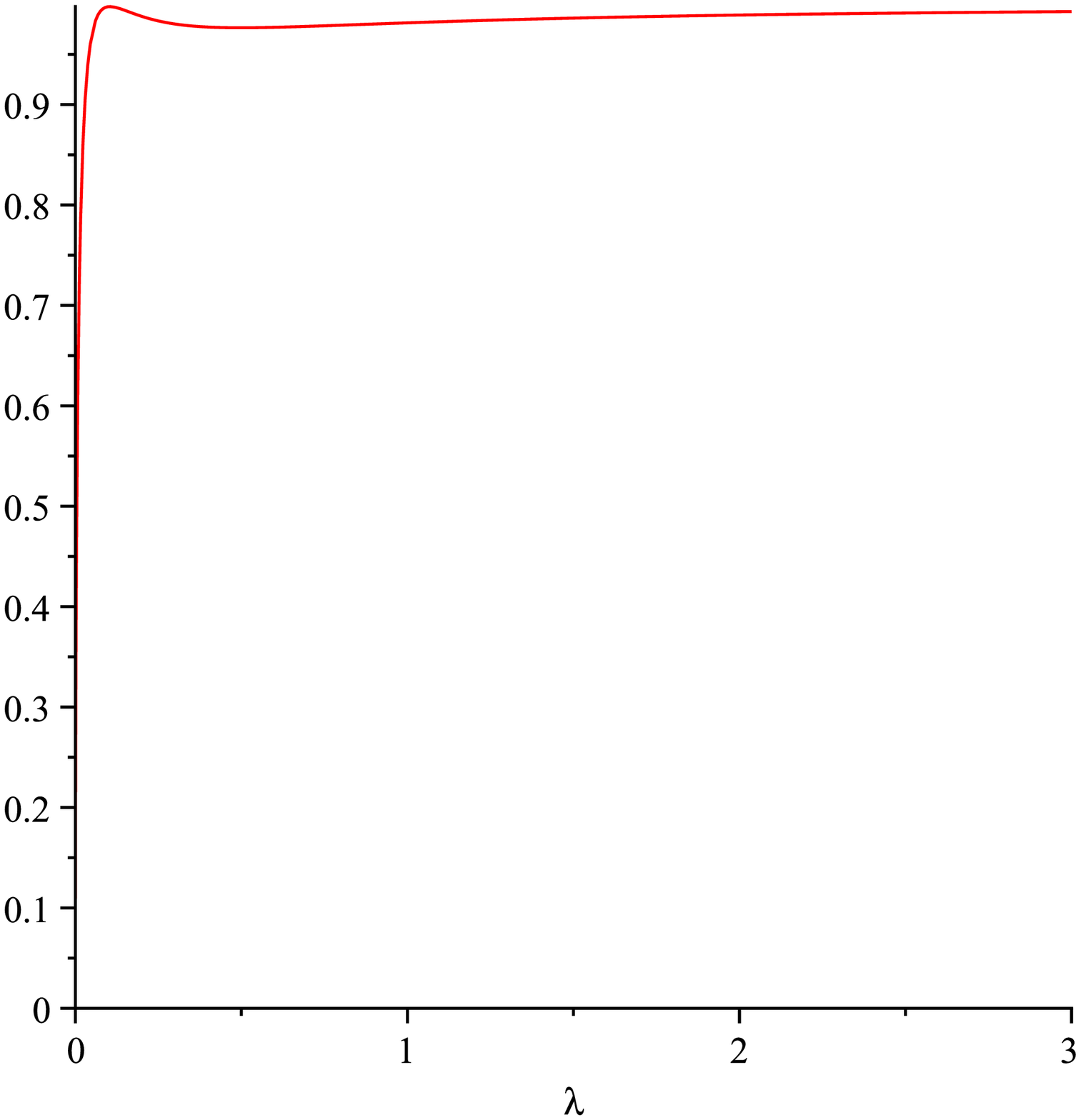,width=4cm,height=5cm,angle=0}
\psfig{file=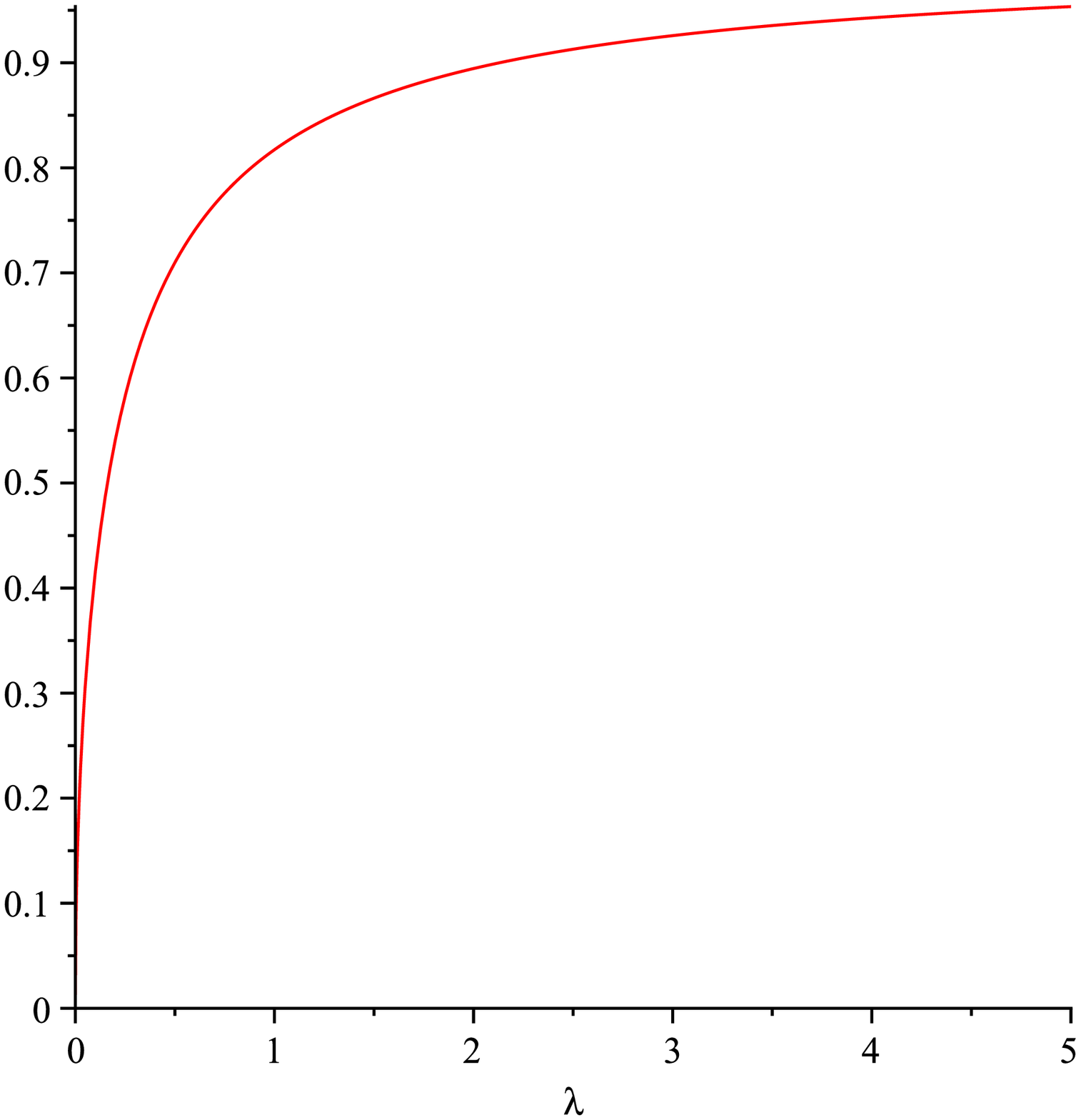,width=4cm,height=5cm,angle=0}
}
\caption{Graphs of $\lambda\to\sqrt{\lambda}\,
\frac{\coth(\pi\sqrt{\beta+\lambda})}
{\sqrt{\beta+\lambda}}$ for $\beta=0.01,\,\beta_*,\,0.5$.}
\label{fig:K1}
\end{figure}

\subsection*{Second inequality.}
Now let $x\in[0,2\pi]$, and  we
consider inequality~\eqref{ineqsubtract},
where $u\in \dot H^1(0,2\pi)_{\mathrm{per}}=H^1(0,2\pi)_{\mathrm{per}}\cap\int_0^{2\pi}u(x)dx=0$, and  $0\le\beta<1$.

The
Green's function $G_\lambda(x,\xi)$ of the operator
$$
\mathbb{A}(\lambda)=-\frac{d^2}{dx^2}-\beta+\lambda,\qquad\lambda>0,
$$
in the space of functions with mean value zero
is as follows
$$
G_\lambda(x,\xi)=\frac1{2\pi}\sum_{k\in\mathbb{Z}_0}
\frac{e^{ik(x-\xi)}}{k^2-\beta+\lambda}\,,\qquad\mathbb{Z}_0=\mathbb{Z}\setminus\{0\},
$$
and
\begin{equation}\label{G2}
G_\lambda(\xi,\xi)=\frac1{2\pi}\sum_{k\in\mathbb{Z}_0}
\frac1{k^2-\beta+\lambda}=:f_\beta(\lambda).
\end{equation}
As before we have the following expression for the sharp constant
$K_2(\beta)$:
 $$
 K_2(\beta)=
2\sup_{\lambda>0}\sqrt{\lambda}f_\beta(\lambda).
$$
Next, from \eqref{formsum} we obviously have
$$
\sum_{k\in\mathbb{Z}_0}\frac1{k^2+\lambda}=
\frac{\pi\sqrt{\lambda}\coth(\pi\sqrt\lambda\,)-1}{{\lambda}},
$$
giving
$$
f_\beta(\lambda)=\frac1{2}\frac{\sqrt{\lambda-\beta}\coth(\pi\sqrt{\lambda-\beta})-1/\pi}
{{\lambda-\beta}}\,,
$$
and, finally,
\begin{equation}\label{Kbeta2}
K_2(\beta)=
\sup_{\lambda>0}\sqrt{\lambda}\,
\frac{\sqrt{\lambda-\beta}\coth(\pi\sqrt{\lambda-\beta})-1/\pi}
{{\lambda-\beta}}\,.
\end{equation}
The existence of the limit
$$
\lim_{\lambda\to\infty}\sqrt{\lambda}\,
\frac{\sqrt{\lambda-\beta}\coth(\pi\sqrt{\lambda-\beta})-1/\pi}
{{\lambda-\beta}}=1,
$$
implies that
$$
K_2(\beta)\ge1.
$$
This time
$K_2(\beta)$ is monotone non-decreasing
$$
\beta_1<\beta_2\quad\Rightarrow\quad f_{\beta_1}(\lambda)<f_{\beta_2}(\lambda)
\quad\Rightarrow\quad K_1(\beta_1)\le K_1(\beta_2),
$$
and, finally,
$$
K_2(\beta_0)=1, \ 0<\beta_0<1 \quad\Rightarrow\quad K_2(\beta)=1\quad \text{for}\quad 0\le\beta\le\beta_0.
$$

The graphs of the function $\lambda\to\sqrt{\lambda}\,
\frac{\sqrt{\lambda-\beta}\coth(\pi\sqrt{\lambda-\beta})-1/\pi}
{{\lambda-\beta}}$ are shown in Fig.~\ref{fig:K2}, which is somewhat
symmetrical to Fig.~\ref{fig:K1}.
For  $\beta$ close to $1$ the graphs have a sharp peak near
the origin. For  $\beta_{**}=0.839\dots$ (again found numerically)
the supremum is equal to $1$ and is attained both at a finite $\lambda$ and
at infinity. Hence $K_1(\beta)=1$
for all $\beta\ge\beta_{**}$.

Thus, we have proved the following theorem.
\begin{theorem}\label{Th:K2}
The sharp constant $K_2(\beta)$ in inequality~\eqref{ineqsubtract}
is given by \eqref{Kbeta2}. Furthermore,
for all $\beta\le\beta_{**}=0.839\dots$, $K_2(\beta)=1$.
\end{theorem}
\begin{remark}\label{R:history}
{\rm
Inequality~\eqref{ineqsubtract} with $\beta=0$ and $K_2(0)=1$
goes back to~\cite{Hardy} where it was used for the proof of the
Carlson inequality. Sharp constants in the higher-order inequalities
were found in~\cite{Taikov} for $x\in\mathbb{R}$ and in \cite{I98JLMS}
for~$x\in\mathbb{S}^1$. Various refinements and improvements of this
type of inequalities were recently obtained in \cite{Zelik}, \cite{IMRN}, \cite{IZ}.
}\end{remark}
\begin{figure}[htb]
\centerline{\psfig{file=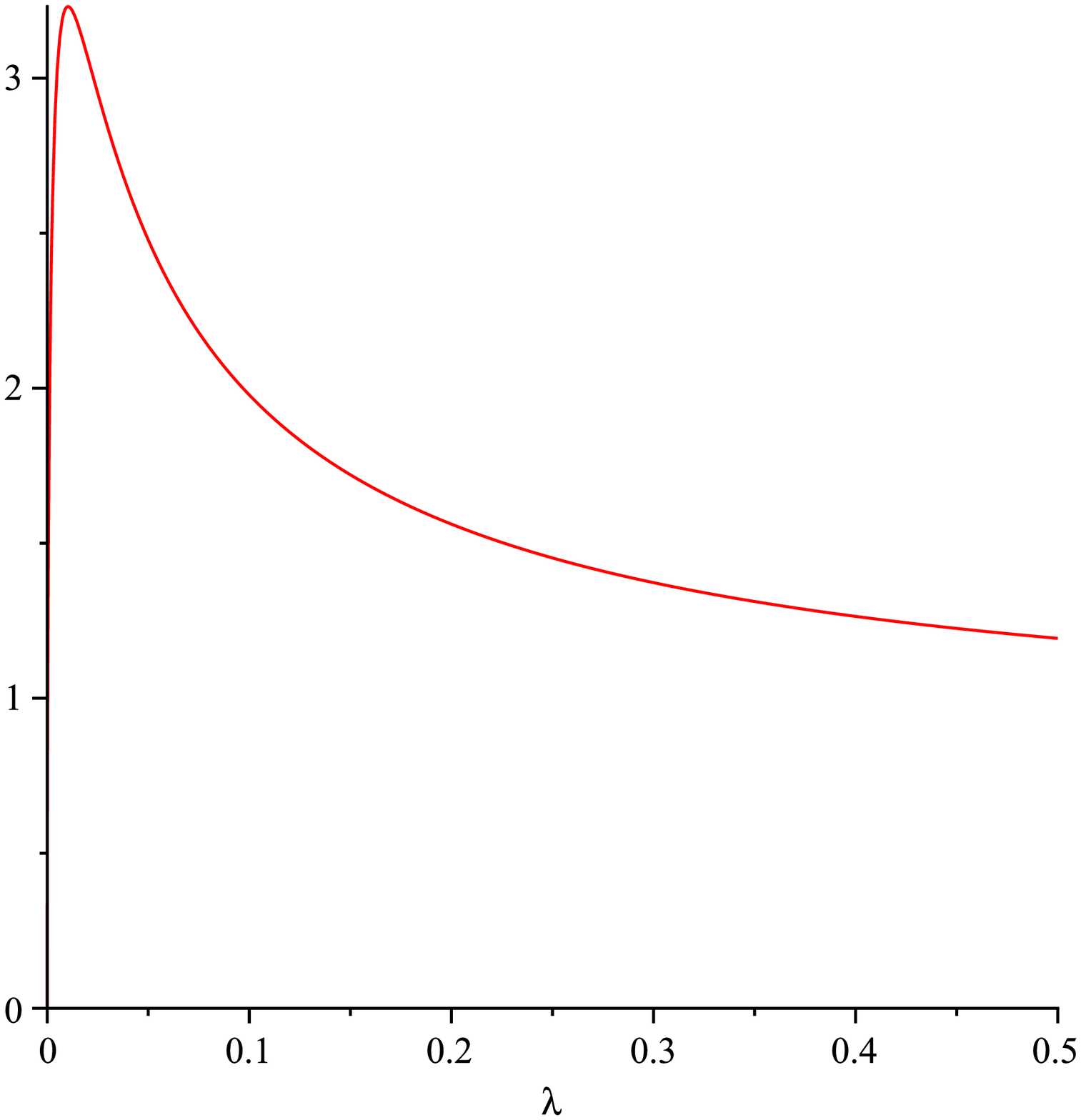,width=4cm,height=5cm,angle=0}
\psfig{file=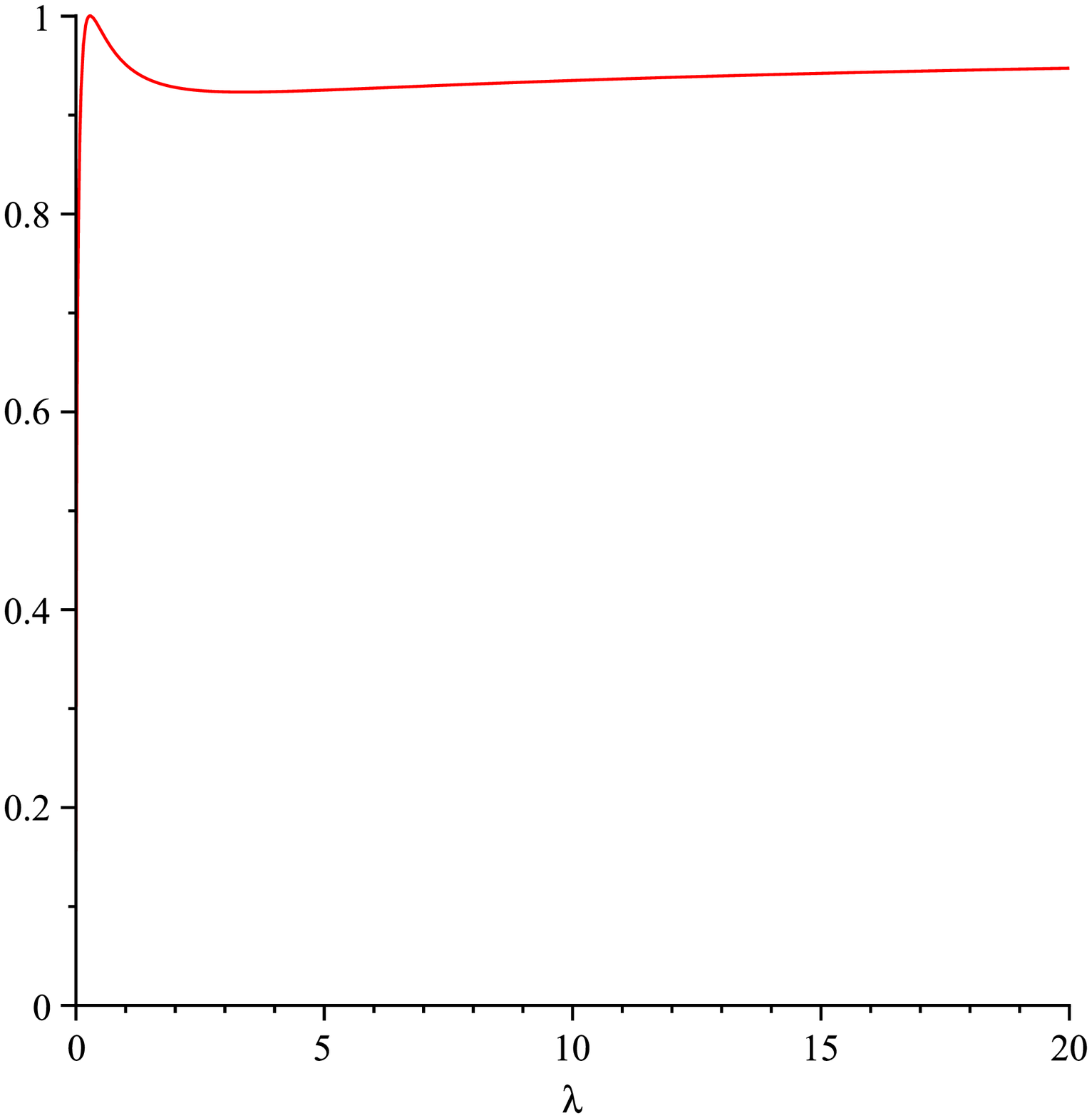,width=4cm,height=5cm,angle=0}
\psfig{file=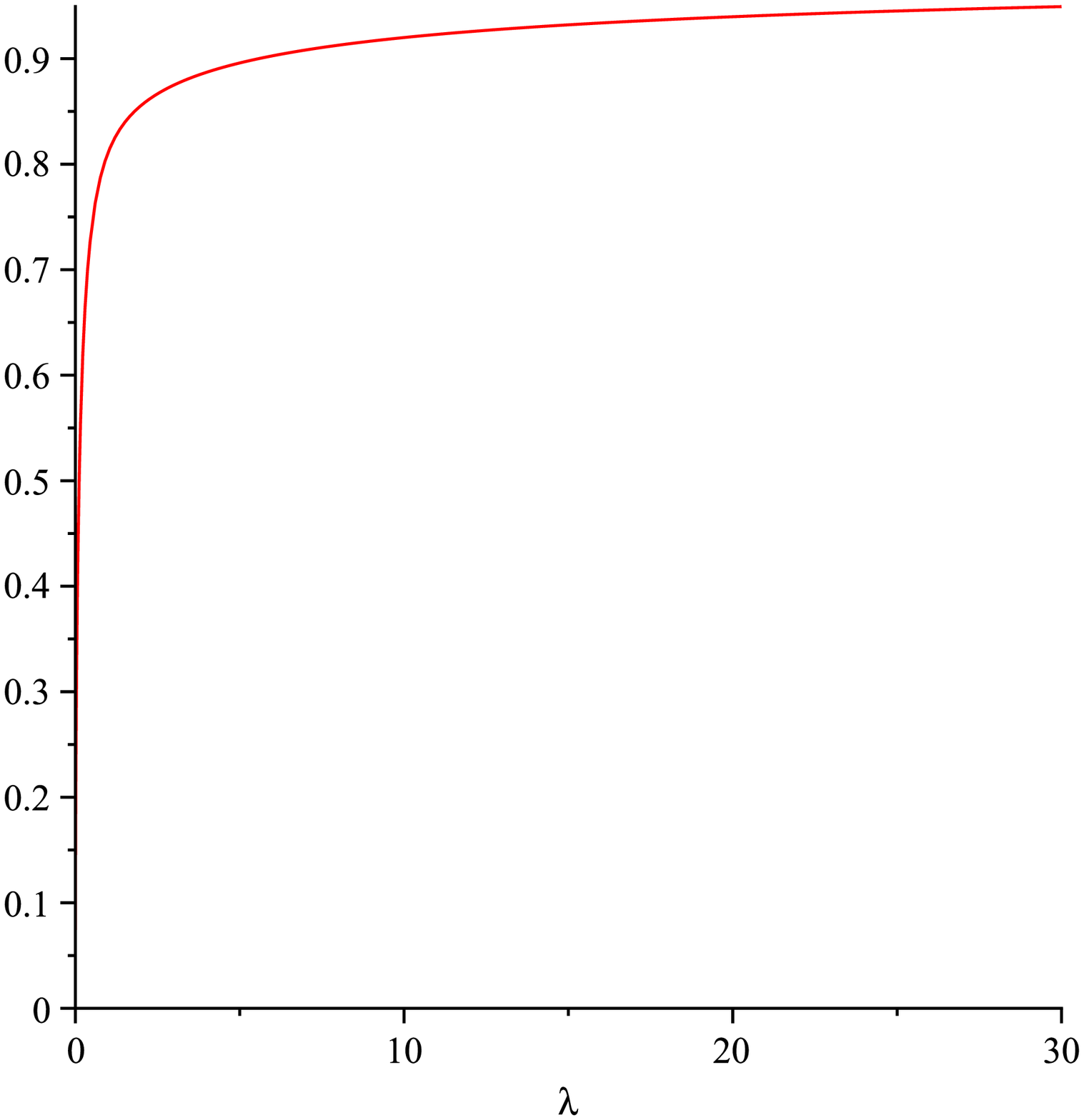,width=4cm,height=5cm,angle=0}
}
\caption{Graphs of $\lambda\to\sqrt{\lambda}\,
\frac{\sqrt{\lambda-\beta}\coth(\pi\sqrt{\lambda-\beta})-1/\pi}
{{\lambda-\beta}}$ for $\beta=0.99,\,\beta_{**},\,0.5$.}
\label{fig:K2}
\end{figure}

\begin{remark}\label{R:series}
{\rm
By the definition of $K_1(\beta)$ and $K_2(\beta)$ we have
\begin{equation}\label{bound}
\aligned
&\sum_{k\in\mathbb{Z}}
\frac1{k^2+\beta+\lambda}\le K_1(\beta)\frac\pi{\sqrt{\lambda}}\,,
\quad\beta>0,\\
&\sum_{k\in\mathbb{Z}_0}
\frac1{k^2-\beta+\lambda}\le K_2(\beta)\frac\pi{\sqrt{\lambda}}\,,
\quad\beta\in[0,1).
\endaligned
\end{equation}
}
\end{remark}

\setcounter{equation}{0}
\section{One-dimensional
 Sobolev inequalities for traces of matrices and associated
 Lieb--Thirring inequalities}\label{sec2}
Let $\{\phi_n\}_{n=1}^N$ be an orthonormal family of periodic
vector-functions
$$
\phi_n(x)=(\phi_n(x,1),\dots,\phi_n(x,M))^T
$$
defined for $x\in[0,2\pi/\alpha]_\mathrm{per}$:
$$
(\phi_n,\phi_m)
=
\sum_{j=1}^M\int_0^{L}\phi_n(x,j)\overline{\phi_m(x,j)}dx
=\int_0^{L} \phi_n(x)^T\overline{\phi_m(x)}dx=\delta_{nm},
$$
where we set for brevity
$$
L:=2\pi/\alpha.
$$

We consider the $M\times M$ matrix $U(x,y)$
\begin{equation}\label{matrixU}
U(x,y)=\sum_{n=1}^N\phi_n(x)\overline{\phi_n(y)}^T
\end{equation}
with entries
$[U(x,y)]_{jk}=\sum_{n=1}^N\phi_n(x,j)\overline{\phi_n(y,k)}$.
Clearly,
$$
U(x,y)^*=U(y,x)
$$
and by orthonormality
$$
\aligned
\int_0^{L} U(x,y)U(y,z)dy=\sum_{n,n'=1}^N\int_0^{L} \phi_n(x)
\overline{\phi_n(y)}^T\phi_{n'}(y)\overline{\phi_{n'}(z)}^Tdy=\\=
\sum_{n=1}^N \phi_n(x)
\overline{\phi_{n}(z)}^T=U(x,z).
\endaligned
$$

\begin{theorem}\label{T:matrix}
Let  $\phi_n(x,j)\in H^1(0,L)_{\mathrm{per}}$. Then
\begin{equation}\label{trace1}
\int_0^{L}\Tr[U(x,x)^{3}]dx\le K_1(\beta)^2\sum_{n=1}^N\sum_{j=1}^M
\int_0^{L}(|\phi'_n(x,j)|^2+\alpha^2\beta|\phi_n(x,j)|^2)dx,
\end{equation}
where $K_1(\beta)$ is defined in~\eqref{ineqadd} and $\beta>0$.

If $\alpha=1$, $L=2\pi$, and  $\int_0^{2\pi} \phi_n(x,j)=0$ for all $n$ and $j$, then
\begin{equation}\label{trace2}
\int_0^{2\pi}\Tr[U(x,x)^{3}]dx\le K_2(\beta)^2\sum_{n=1}^N\sum_{j=1}^M
\int_0^{2\pi}(|\phi'_n(x,j)|^2-\beta|\phi_n(x,j)|^2)dx,
\end{equation}
where $K_2(\beta)$ is defined in~\eqref{ineqsubtract} and $\beta\in[0,1)$.
\end{theorem}
\begin{proof}
Both inequalities are proved similarly. The argument, in turn, follows very closely
the proof of Theorem~6.1 in~\cite{IMRN}.
Let us prove, for example, the second inequality~\eqref{trace2}.
We have
$$
\tilde U(n,x)  = \frac1{\sqrt{2\pi}}\int_0^{2\pi}   e^{- i  y n} U(y,x) dy,\quad n\in \mathbb{Z}_0,
$$
so that
$$
U(y,x) =\frac1{\sqrt{2\pi}}\sum_{k\in\mathbb{Z}_0}  {e^{ i  y k}}  \tilde U(k,x).
$$
Next,
\begin{equation}\label{orthint}
\sum_{k\in\mathbb{Z}_0} \tilde U(k,x)^* \tilde U(k,x) =
\int_0^{2\pi}  U(y,x)^* U(y,x) dy  = U(x,x),
\end{equation}
 and
$$
\aligned
\sum_{k\in\mathbb{Z}_0} (k^2-\beta)
\tilde U(k,x)^* \tilde U(k,x) =
\int_0^{2\pi} [\partial_y U (y,x)]^* \partial_y U(y,x)  dy-\\
-\beta\int_0^L U (y,x)^* U(y,x)dy,
\endaligned
$$
where the last term is equal to $-\beta U (x,x)$.
Now, by orthonormality
\begin{equation}\label{Dtrace}
\allowdisplaybreaks
\aligned
&\Tr\left[\int_0^{2\pi} \sum_{k\in \mathbb{Z}_0} (k^2-\beta)
\tilde U(k,x)^* \tilde U(k,x)  dx \right]=\\
&\Tr\left[\int_0^{2\pi} \int_0^{2\pi} \sum_{n,n'=1}^{N}
\phi_{n'}'(y)\overline{\phi_{n'}(x)}^T
\phi_n(x)\overline{\phi_n'(y)}^T
 dx dy-\beta\int_0^{2\pi}U(x,x)dx\right]=\\
&\Tr\left[ \int_0^{2\pi} \sum_{n=1}^{N}\left[
\phi_{n}'(y)\overline{\phi_n'(y)}^T-\beta\phi_{n}(y)\overline{\phi_n(y)}^T\right] dy\right]=
\\
&\qquad=\sum_{n=1}^N\sum_{j=1}^M
\int_0^{2\pi}(|\phi'_n(x,j)|^2-\beta|\phi_n(x,j)|^2)dx.
\endaligned
\end{equation}

Now consider
$$
\aligned
&\Tr[U(x,x)^{3}] = \sum_{k\in\mathbb{Z}_0}
\Tr[U(x,x)^{2} \tilde U(k,x) ]  \frac {e^{ i  x k}}{\sqrt{2\pi}}=\\
&\sum_{k\in\mathbb{Z}_0}
\Tr\biggl [ [(k^2-\beta) I + \Lambda(x)^{2}]^{-1/2} U(x,x)^{2}
\tilde U(k,x)
[(k^2-\beta) I + \Lambda(x)^{2}]^{1/2}\biggr ]
\frac{e^{ i  x k}}{\sqrt{2\pi}}\,,
\endaligned
$$
where $\Lambda(x)$ is an arbitrary positive definite matrix. Using
 below the  Cauchy--Schwartz inequality for
matrices  we get the upper bounds
\begingroup
\allowdisplaybreaks
\begin{align*}
&\sqrt{2\pi}\Tr[U(x,x)^{3}] \le\\
&\sum_{k\in\mathbb{Z}_0}
\bigl|\Tr\bigl[ [(k^2-\beta) I + \Lambda(x)^{2}]^{-1/2} U(x,x)^{2}
\tilde U(k,x)[(k^2-\beta) I + \Lambda(x)^{2}]^{1/2}\bigl ]
\bigl|\le\\
 &\sum_{k\in\mathbb{Z}_0}
\bigl(\Tr\bigl [ U(x,x)^{2}[(k^2-\beta) I + \Lambda(x)^{2}]^{-1}
U(x,x)^{2}\bigr]\bigr)^{1/2}\times\\
&\bigl(\Tr\bigl[
[(k^2-\beta) I + \Lambda(x)^{2}]
\tilde U(k,x)^*\tilde U(k,x)\bigr]\bigr)^{1/2}\le\\
 &\left(\sum_{k\in\mathbb{Z}_0}
\Tr\left [ U(x,x)^{2}[(k^2-\beta) I + \Lambda(x)^{2}]^{-1}
U(x,x)^{2}\right]\right)^{1/2}\times\\
&\left(\sum_{k\in\mathbb{Z}_0}\Tr\left[
[(k^2-\beta) I + \Lambda(x)^{2}]
\tilde U(k,x)^*\tilde U(k,x)\right]\right)^{1/2}\,.
\end{align*}
\endgroup

We now  use the  matrix inequality
\begin{equation}\label{matrineq}
\sum_{k\in\mathbb{Z}_0}[(k^2-\beta) I + \Lambda(x)^{2}]^{-1}<
K_2(\beta)\pi\Lambda(x)^{-1},
\end{equation}
which follows from the scalar inequality (see~\eqref{bound})
$$
\sum_{k\in\mathbb{Z}_0}\frac1{k^2-\beta  + \lambda^{2}}<
K_2(\beta)\pi\lambda^{-1}
$$
applied to each eigenvector
$e=e(x)$ of $\Lambda(x)$ with eigenvalue
$\lambda=\lambda(x)>0$ from the orthonormal basis
$\{e_j(x),\ \lambda_j(x)\}_{j=1}^M$.
This gives for  the first factor
$$
\allowdisplaybreaks
\aligned
\sum_{k\in\mathbb{Z}_0}
\Tr\left [ U(x,x)^{2}[(k^2-\beta) I + \Lambda(x)^{2}]^{-1}
U(x,x)^{2}\right]=\\
\Tr\left [ U(x,x)^{2}\sum_{k\in\mathbb{Z}_0}[(k^2-\beta) I + \Lambda(x)^{2}]^{-1}
U(x,x)^{2}\right]\le\\
K_2(\beta)\pi\Tr\left [ U(x,x)^{2}\Lambda(x)^{-1}
U(x,x)^{2}\right].
\endaligned
$$

For the second factor we simply have
$$
\aligned
\sum_{k\in\mathbb{Z}_0}\Tr\left[
[(k^2-\beta) I + \Lambda(x)^{2}]
\tilde U(k,x)^*\tilde U(k,x)\right]=\\
\sum_{k\in\mathbb{Z}_0}\Tr
\left[(k^2-\beta) \tilde U(k,x)^*\tilde U(k,x)\right] +
\sum_{k\in\mathbb{Z}_0}\Tr\left[\Lambda(x)^{2}
\tilde U(k,x)^*\tilde U(k,x)\right].
\endaligned
$$

If we now chose $\Lambda(x)=\gamma(U(x,x)+\varepsilon I)$,
$\gamma>0$,  and let $\varepsilon\to0$
we obtain (observing that $\lambda^{4}/(\lambda+\varepsilon)\to
\lambda^{3}$ as $\varepsilon\to0$ for $\lambda\ge0$; this is required in case when $U(x,x)$
is not invertible)
$$
\aligned
\sqrt{2\pi}\Tr[U(x,x)^{3}] \le\sqrt{\pi{K_2(\beta)}}\,
\gamma^{-1/2}\Tr[U(x,x)^{3}]^{1/2}\times\\
\left(
\sum_{k\in\mathbb{Z}_0}\Tr
\left[(k^2-\beta) \tilde U(k,x)^*\tilde U(k,x)\right]
+\gamma^{2}\Tr[U(x,x)^{3}]
\right),
\endaligned
$$
where we have also used~\eqref{orthint}, or
$$
\aligned
&\Tr[U(x,x)^{3}] \le\\
\frac{K_2(\beta)}2
\biggl(\gamma^{-1}
\sum_{k\in\mathbb{Z}_0}
&\Tr
\left[(k^2-\beta) \tilde U(k,x)^*
\tilde U(k,x)\right]
+\gamma\Tr[U(x,x)^{3}]
\biggr).
\endaligned
$$
If we optimize over $\gamma$, we obtain
$$
\aligned
\Tr[U(x,x)^{3}]\! \le\!
K_2(\beta)\biggl(\Tr[U(x,x)^{3}]\biggr)^{1/2}\!
\biggl(
\sum_{k\in\mathbb{Z}_0}
\Tr
\left[(k^2-\beta) \tilde U(k,x)^*
\tilde U(k,x)\right]
\biggr)^{1/2}
\endaligned
$$
or
$$
\Tr[U(x,x)^{3}] \le
K_2(\beta)^2
\sum_{k\in\mathbb{Z}_0}
\Tr
\left[(k^2-\beta) \tilde U(k,x)^*
\tilde U(k,x)\right].
$$
If we integrate with respect to $x$ and use~\eqref{Dtrace},
 we obtain~\eqref{trace2}.
 The proof of the theorem is complete since the proof of ~\eqref{trace1}
 is totally similar.
\end{proof}

\begin{remark}\label{R:E-F}
{\rm
In the scalar case $M=1$ inequalities \eqref{trace1}, \eqref{trace2} go over to
\begin{equation}\label{EF}
\aligned
&\int_0^{L}\biggl(\sum_{n=1}^N|\phi_n(x)|^2\biggr)^3dx\le
K_1(\beta)^2\sum_{n=1}^N \int_0^{L}
(|\phi_n'(x)|^2+\alpha^2\beta|\phi_n(x)|^2)dx,\\
&\int_0^{2\pi}\biggl(\sum_{n=1}^N|\phi_n(x)|^2\biggr)^3dx\le
K_2(\beta)^2\sum_{n=1}^N \int_0^{2\pi}
(|\phi_n'(x)|^2-\beta|\phi_n(x)|^2)dx
\endaligned
\end{equation}
and follow from  interpolation inequalities \eqref{ineqadd},\,\eqref{ineqsubtract}
by the method  of~\cite{E-F}.
}
\end{remark}

Let us now consider two one-dimensional Schr\"odinger operators
with periodic boundary conditions and matrix-valued
potentials~\eqref{H1} and \eqref{H2}, where $\mathrm{P}$ is the
orthogonal projection
\begin{equation}\label{P}
(\mathrm{P}\psi)(x)=\psi(x)-\frac1{2\pi}\int_0^{2\pi}\psi(t)dt
\end{equation}
acting component-wise.

Inequalities \eqref{trace1} and \eqref{trace2} in Theorem~\ref{T:matrix} are
 equivalent to the estimate of the $1$-moments of
the negative eigenvalues $-\lambda_j\le0$ and $-\mu_j\le0$ of the Schr\"odinger  operators
\eqref{H1} and \eqref{H2}, respectively.
\begin{theorem}\label{Th:matrS1}
Let $V$ be a non-negative $M\times M$ Hermitian matrix such that $\Tr
V^{3/2} \in L_1$. Then the operators~\eqref{H1} and \eqref{H2} have
discrete spectrum, and their  negative eigenvalues $\{-\lambda_j\}$ and $\{-\mu_j\}$
 satisfy the estimates~\eqref{neg1} and \eqref{neg2}, respectively.
\end{theorem}
\begin{proof} Once we have the inequalities for the traces, the proof is completely analogous
to the proof of Theorem~1 in \cite{D-L-L} (see also Theorem~6.3 in \cite{IMRN}.), and the assertion of the theorem
is just the inequality $\mathrm{L}_{1,1}\le\frac2{3\sqrt{3}}\sqrt{\mathrm{k}_1}$ in the one dimensional
matrix case, see~\eqref{k=L}. We include the proof of~\eqref{neg2} for the sake of completeness
(the proof of \eqref{neg1} is completely similar).

 Let $\{\phi_n\}_{n=1}^N$
be the orthonormal eigen-vector functions corresponding
to the negative eigenvalues $\{-\mu_n\}_{n=1}^N$ of~\eqref{H2} :
$$
-\phi''_n-\delta\phi_n-V\phi_n=-\mu_n\phi_n.
$$
Using~\eqref{trace2} and H\"older's inequality for
 traces  we obtain
$$
\aligned
&\sum_{n=1}^N\mu_n=\\
&-\sum_{n=1}^N \sum_{j=1}^M  \int_0^{2\pi} (|\phi'_n(x,j)|^2-\delta|\phi_n(x,j)|^2) dx
+\int_0^{2\pi}\Tr [V(x)U(x,x)]dx\le\\
&\le-K_2(\delta)^{-2}X+
\left(\int_0^{2\pi}\Tr[V(x)^{3/2}]dx\right)^{2/3}
X^{1/3},
\endaligned
$$
where  $X:=\int_0^{2\pi}\Tr [U(x,x)^3]dx$.
Calculating the maximum with respect to $X$ we
obtain~\eqref{neg2}.
\end{proof}

 We observe that in terms of~\eqref{class} the estimates \eqref{neg1} and \eqref{neg2} can be written in the form
\begin{align}
\sum_j\lambda_j\le\frac \pi{\sqrt{3}}K_1(\beta)\mathrm{L}^\mathrm{cl}_{1,1}
\int_0^{L}  \Tr[V(x)^{3/2}]dx,\label{neg1cl}\\
\sum_j\mu_j\le\frac \pi{\sqrt{3}}K_2(\beta)\mathrm{L}^\mathrm{cl}_{1,1}
\int_0^{2\pi}  \Tr[V(x)^{3/2}]dx.\label{neg2cl}
\end{align}

By using the Aizenmann-Lieb argument \cite{AisLieb} we
immediately obtain the following estimates for the Riesz means of order $\gamma\ge1$
of the eigenvalues of the operators~\eqref{H1} and \eqref{H2}.
\begin{corollary}\label{1D gamma-moments}
Let $V\ge0$ be a  $M\times M$ Hermitian matrix, such that
$\Tr V^{\gamma + 1/2} \in L_1$.
Then for any $\gamma \ge 1$  the negative eigenvalues of
the operators \eqref{H1}, \eqref{H2} satisfy the inequalities
\begin{align}
\sum_j\lambda_j^\gamma\le\frac \pi{\sqrt{3}}K_1(\beta)\mathrm{L}^\mathrm{cl}_{\gamma,1}
\int_0^{L}  \Tr[V(x)^{1/2+\gamma}]dx,\label{neg1AZ}\\
\sum_j\mu_j^\gamma\le\frac \pi{\sqrt{3}}K_2(\beta)\mathrm{L}^\mathrm{cl}_{\gamma,1}
\int_0^{2\pi}  \Tr[V(x)^{1/2+\gamma}]dx.\label{neg2AZ}
\end{align}
\end{corollary}

\begin{proof}It is enough to prove this result for smooth matrix-valued potentials.
Furthermore, we consider only \eqref{neg1AZ}, the treatment of the second operator
being completely similar.
Note that Theorem \ref{Th:matrS1} is equivalent to
$$
\sum_n\lambda_n\le\frac 2{3\sqrt{3}}K_1(\beta)
\frac1{\mathrm{L}_{1,1}^{\mathrm{cl}}}  \,\int_0^{L}
\int_{-\infty}^\infty \Tr\left[\left(|\xi|^2 - V(x)\right)_-\right]\, \frac{d\xi dx}{2\pi}.
$$
Scaling gives the simple identity for all $s\in \mathbb R$
$$
s_-^\gamma = C_\gamma \, \int_0^\infty t^{\gamma-2} (s+t)_- dt ,
 \qquad C_\gamma^{-1} = \mathcal B(\gamma - 1, 2),
$$
where $\mathcal B$ is the Beta function. Let
$\{\nu_j(x)\}_{j=1}^M $ be the eigenvalues of the matrix-function $V(x)$.
Then
 \begin{align*}
& \sum_n \lambda_n^\gamma = C_\gamma \,
\sum_n \int_0^\infty t^{\gamma-2} (-\lambda_n + t)_- dt\\
& \le \frac{2K_1(\beta)}{3\sqrt{3}} \, \frac{C_\gamma}{\mathrm{L}_{1,1}^{\mathrm{cl}}} \,
\int_0^\infty \int_0^{L} \int_{-\infty}^\infty  t^{\gamma-2}
\Tr\left[\left(|\xi|^2 - V(x) + t\right)_-\right]\,\frac{d\xi dx}{2\pi} \,  dt \\
& = \frac{2K_1(\beta)}{3\sqrt{3}}\, \frac{C_\gamma}{\mathrm{L}_{1,1}^{\mathrm{cl}}} \,
\sum_{j=1}^M\, \int_0^\infty \int_0^{L} \int_{-\infty}^\infty  t^{\gamma-2}
 \Tr\left[\left(|\xi|^2 - \nu_j(x) + t\right)_-\right]\,\frac{d\xi dx}{2\pi} \,  dt\\
&= \frac{2K_1(\beta)}{3\sqrt{3}}\frac1{\mathrm{L}_{1,1}^{\mathrm{cl}}} \, \int_0^{L}
\int_{-\infty}^\infty   \Tr \left[\left(|\xi|^2 - V(x)\right)_-^\gamma\right]\,
\frac{d\xi dx}{2\pi} \\
& = \frac {2K_1(\beta)}{3\sqrt{3}}
\frac{\mathrm{L}_{\gamma,1}^{\mathrm{cl}}}{\mathrm{L}_{1,1}^{\mathrm{cl}}}\,
\int_0^{L}  \Tr[V(x)^{1/2 + \gamma}] \, dx,
\end{align*}
which completes the proof since $\mathrm{L}_{1,1}^{\mathrm{cl}}=2/(3\pi)$.
\end{proof}

\setcounter{equation}{0}
\section{Lieb--Thirring inequalities on the torus}\label{sec3}

In this section we consider Lieb--Thirring inequalities on the torus
\eqref{T2a}, \eqref{alpha},
paying special attention to the two-dimensional case.

We recall the orthogonal  projection $\mathrm{P}$ defined in~\eqref{projP1},
so that the resulting function  has zero average  with respect to shortest
coordinate  $x_d$ for all $x_1,\dots,x_{d-1}$.

\begin{theorem}\label{T:ddim}
Let $V\ge0$, $\gamma\ge1$,  and let $V\in L_{\gamma+d/2}(\mathbb{T}^d_\alpha)$.
Then the negative eigenvalues $-\lambda_n\le0$ of the  the Schr\"odinger operator
\begin{equation}\label{Schrod}
\mathcal{H}=-\Delta-\mathrm{P}(V(x)\mathrm{P}\,\cdot)
\end{equation}
 satisfy the bound
\begin{equation}\label{LTdest}
\sum_n\lambda_n^\gamma\le \mathrm{L}_{\gamma,d}\int_{\mathbb{T}^d_\alpha}V^{\gamma+d/2}(x)dx
\end{equation}
where
\begin{equation}\label{LTdconst}
\mathrm{L}_{\gamma,d}\le\left(\frac\pi{\sqrt{3}}\right)^d\,\prod_{j=1}^{d-1}
K_1(\beta_j)\,K_2(\delta)
\mathrm{L}^\mathrm{cl}_{\gamma,d},
\end{equation}
provided that
$\beta_j>0$, $j=1,\dots d-1$ are chosen so small
that
$$
\delta:=\alpha_1^2\beta_1+\dots+\alpha_{d-1}^2\beta_{d-1}<1.
$$
If
\begin{equation}\label{betadelta}
\beta_j\ge\beta_*\quad\&\quad\delta\le\beta_{**},
\end{equation}
(where the numbers $\beta_*$ and $\beta_{**}$ are defined in
Section~\ref{sec1}), then
\begin{equation}\label{LTdconst1}
\mathrm{L}_{\gamma,d}\le\left(\frac\pi{\sqrt{3}}\right)^d\,
\mathrm{L}^\mathrm{cl}_{\gamma,d}.
\end{equation}
For
$$
d\le\biggl[\frac{\beta_{**}}{\beta_{*}}\biggr]+1=19
$$
 condition~\eqref{betadelta} can be satisfied
 for any $\alpha$, so that~\eqref{LTdconst1} always holds.
\end{theorem}
\begin{proof}
We write the operator $\mathcal{H}$ in the form
$$
\left(-\frac{d^2}{dx_1^2}+\alpha_1^2\beta_1\right)+\dots+
\left(-\frac{d^2}{dx_{d-1}^2}+\alpha_{d-1}^2\beta_{d-1}\right)+
\left(-\frac{d^2}{dx_d^2}-\delta\right)-\mathrm{P}(V(x)\mathrm{P}\,\cdot).
$$

We use the lifting argument with respect to dimensions developed in~\cite{Lap-Weid}.
More precisely, we apply estimate \eqref{neg1AZ} $d-1$ times with respect to
variables $x_1,\dots,x_{d-1}$, so that $\gamma$ is increased by $1/2$ at
each step, and, finally, we use~\eqref{neg2AZ} with respect to $x_d$. Using the variational principle and
denoting the negative pars of the operators by $[\,\cdot\,]_-$ we obtain
\begingroup
\allowdisplaybreaks
\begin{align*}
&\sum_n\lambda_n^\gamma(\mathcal{H})=\\=
 &\sum_n\lambda_n^\gamma\biggl(-\partial_1^2-\alpha_1^2\beta_1+
 \sum\nolimits_{j=2}^{d-1}(-\partial_j^2-\alpha_1^2\beta_1)
 +(-\partial_d^2+\delta-\mathrm{P}(V(x)\,\cdot))\biggr)\le\\\le
 &\sum_n\lambda_n^\gamma\biggl(-\partial_1^2-\alpha_1^2\beta_1+\\
 &\qquad+\left[\sum\nolimits_{j=2}^{d-1}(-\partial_j^2-\alpha_1^2\beta_1)
 +(-\partial_d^2+\delta-\mathrm{P}(V(x)\,\cdot))\right]_-\biggr)\le\\\le
&\frac\pi{\sqrt{3}}K_1(\beta_1)\mathrm{L}^\mathrm{cl}_{\gamma,1}\times\\
&\qquad\int_0^{L_1}
\Tr\left[\sum\nolimits_{j=2}^{d-1}(-\partial_j^2-\alpha_1^2\beta_1)
+(-\partial_d^2+\delta-\mathrm{P}(V(x)\,\cdot))\right]_-^{\gamma+1/2}dx_1\le\\
 &\dots\dots\dots\\\le
&\left(\frac\pi{\sqrt{3}}\right)^{d-1}\prod_{j=1}^{d-1}(K_1(\beta_j)
\prod_{j=1}^{d-1}\mathrm{L}^\mathrm{cl}_{\gamma+(j-1)/2,1}\times\\
&\qquad\int_0^{L_1}\dots
\int_0^{L_{d-1}}\Tr\bigl[(-\partial_d^2+\delta-\mathrm{P}(V(x)\,\cdot))\bigr]_-^{\gamma+(d-1)/2}
dx_1\dots dx_{d-1}\le\\\le
&\left(\frac\pi{\sqrt{3}}\right)^{d}\prod_{j=1}^{d-1}K_1(\beta_j)
\prod_{j=1}^{d}\mathrm{L}^\mathrm{cl}_{\gamma+(j-1)/2,1}K_2(\delta)
\int_{\mathbb{T}^d_\alpha}V^{\gamma+d/2}(x)dx,
\end{align*}
\endgroup
which proves \eqref{LTdest}, \eqref{LTdconst} since (see~\eqref{class})
$$
\prod_{j=1}^{d}\mathrm{L}^\mathrm{cl}_{\gamma+(j-1)/2,1}=\mathrm{L}^\mathrm{cl}_{\gamma,d}.
$$
Finally, if we take $\beta_j=\beta_*=0.045\dots$, then $K_1(\beta_j)=1$.
Since $\alpha_j\le1$, it follows that
$$
\delta=\sum_{j=1}^{d-1}\alpha_j^2\beta_j=\beta_{*}\sum_{j=1}^{d-1}\alpha_j^2
\le(d-1)\beta_*.
$$
Therefore the condition $\delta\le\beta_{**}=0.839\dots$
(giving $K_2(\delta)=1$) is always satisfied
if
$$
d-1\le\left[\frac{\beta_{**}}{\beta_*}\right]=18.
$$
The proof is complete.
\end{proof}
\begin{remark}\label{R:d-to-d-1}
{\rm The complementary projection $\mathrm{Q}=I-\mathrm{P}$ maps
$\dot L_2(\mathbb{T}^d_\alpha)$ onto $\dot
L_2(\mathbb{T}^{d-1}_{\alpha'})$, where
$\alpha'=(\alpha_1,\dots,\alpha_{d-1})$ and
$\dot L_2(\mathbb{T}^{d-1}_{\alpha'})$ is the subspace of
 $\dot L_2(\mathbb{T}^d_\alpha)$ of functions
 independent of $x_d$.
}
\end{remark}

\begin{remark}\label{R:Rd-Td}
{\rm
If we compare~\eqref{DLL} and \eqref{LTdconst1},  we see that
the factor $\pi/\sqrt{3}$ accumulates in~\eqref{LTdconst1} at each iteration
with respect to the dimension; while in~\eqref{DLL} already
at the second iteration sharp bounds \cite{Lap-Weid} for the $\gamma$-moments
with $\gamma\ge\frac32$ are available,  which are not known in the
periodic case.
}
\end{remark}

We single out the following corollary that is important for applications.
Let $d=2$ so that
$\mathbb{T}^2_\alpha=(0,2\pi/{\alpha})\times(0,2\pi)$, $\alpha\le1$.

\begin{corollary}\label{Cor:2dtorus}
Let $V\ge0$ and let $V\in L_{2}(\mathbb{T}^2_\alpha)$. Then
the negative eigenvalues $-\lambda_j\le0$ of the operator
\eqref{Schrod} satisfy the following bound:
\begin{equation}\label{LTT221}
\sum_j\lambda_j\le
\left(\frac\pi{\sqrt{3}}\right)^2
\mathrm{L}_{1,2}^{\mathrm{cl}}\int_{\mathbb{T}^2_\alpha}V^2(x)dx=
\frac\pi{24}\int_{\mathbb{T}^2_\alpha}V^2(x)dx.
\end{equation}
Equivalently (see \eqref{L-T-orth-orig}, \eqref{k=L}),  if a family $\{\varphi_j\}_{j=1}^N\in \mathrm{P} H^1(\mathbb{T}^2_\alpha)$ is
orthonormal, then
$\rho_\varphi(x):=\sum_{j=1}^N\varphi_j(x)^2$ satisfies
\begin{equation}\label{L-T-T2}
\int_{\mathbb{T}^2_\alpha}\rho_\varphi(x)^{2}dx\le
\frac\pi 6\sum_{j=1}^N\|\nabla\varphi_j\|^2,
\end{equation}

\end{corollary}

\setcounter{equation}{0}
\section{Applications to attractors}\label{sec4}

Starting from the important paper~\cite{Lieb} the Lieb--Thirring inequalities
play an essential role in the theory of attractors for infinite dimensional
dissipative dynamical systems, especially, for the Navier--Stokes system.
They are used for the first moments ($\gamma=1$) in the equivalent
formulation in terms of orthonormal families.

We first consider the square torus $\mathbb{T}^2=(0,L)\times(0,L)$.

\begin{proposition}\label{T:square}
If $\{v_j\}_{j=1}^m\in \dot  H^1(\mathbb{T}^2)$ is an orthonormal family of vector  functions and $\div v_j=0$, then
$
\rho_v(x):=\sum_{j=1}^m|v_j(x)|^2
$
satisfies
\begin{equation}\label{rhov}
\int_{\mathbb{T}^2}\rho_v(x)^2dx\le  c_{\mathrm{LT}}\sum_{j=1}^m\|\nabla v_j\|^2,
\qquad c_{\mathrm{LT}}\le \frac32.
\end{equation}
\end{proposition}
\begin{proof} (See~\cite{I12JST}.) In the scalar case this follows from
\eqref{T2} and \eqref{L-T-orth-orig}, \eqref{k=L}. In two dimensions the passage
to the vector case with $\div v_j=0$ does not increase the constant \cite{I-MS2005}.
\end{proof}

Turning to the applications we consider the damped and driven Navier--Stokes system
\begin{equation}\label{DDNSE}
\begin{cases}
\partial_t u+\sum_{i=1}^2u^i\partial_i u=\nu\Delta u-\mu u-\nabla p+g,\\
\div u=0,\ \ \ u\big|_{t=0}=u_0.
\end{cases}
\end{equation}
in a periodic square domain
$\mathbb{T}^2=(0,L)\times(0,L)$.
 We  assume that  $g$ and $u$ have mean values zero.
The system is studied in the small viscosity limit $\nu\to0^+$,
while the drag/damping coefficient $\mu>0$ is arbitrary but fixed.

Using the standard notation in the theory of the Navier--Stokes equations
we denote by $H$ the closure
in
$L_2(\mathbb{T}^2)$ of the set of trigonometric polynomials
with divergence and  mean value zero. The norm $\|\cdot\|$
and scalar product
$(\cdot\,,\cdot)$ in $H$ are those of $L_2(\mathbb{T}^2)$.
We project  the first equation onto $H$ and  obtain
the functional evolution equation
\begin{equation}
\partial_tu+B(u,u)+\nu Au=-\mu u+g, \qquad u(0)=u_0,
\label{FNSE}
\end{equation}
where $A$ is the Stokes operator and
$B(u,v)$ is the nonlinear term defined as follows
$$
\langle Au,v\rangle=(\nabla u,\nabla v), \qquad u,v \in H^1\cap H,
$$
and
\begin{equation}\label{trilinear-form}
\langle B(v,u),w\rangle=\int_{\mathbb{T}^2}
\sum_{i,k=1}^2v^k\partial_ku^iw^i\,dx=:b(v,u,w),
\end{equation}
for all $u,v,w\in H^1\cap H$.

Equation~(\ref{FNSE}) has a unique solution $u(t)$ and
the solution semigroup $S_tu_0\to u(t)$
is well defined. The semigroup $S_t$ has a global attractor $\mathcal{A}$
which is a compact strictly invariant set in $H$ attracting under
the action of $S_t$ all bounded sets as $t\to\infty$.
These facts are well known for the classical Navier--Stokes equations
\cite{B-V},\cite{CF88},\cite{Lad},\cite{T};
the case $\mu>0$ is similar.
The solution semigroup $S_t$ is uniformly
 differentiable in $H$
with differential $L(t,u_0):\xi\to U(t)\in H$,
where $U(t)$ is the solution of the variational equation
\begin{equation}\label{var-eq}
\partial_tU=-\nu AU-\mu U\ -
B(U,u(t))-B(u(t),U)=:{\mathcal L}(t,u_0)U, \qquad U(0)=\xi.
\end{equation}
Furthermore, the differential $L(t,u_0)$ depends continuously
on the initial point $u_0\in\mathcal{A}$~\cite{B-V}.

We estimate the numbers
$q(m)$, that is, the sums of the first $m$ global Lyapunov exponents:
\cite{C-F85},\cite{CF88},\cite{T}:
\begin{equation}\label{defqm}
q(m):=\limsup_{t\to\infty}\ \sup_{u_0\in {\mathcal A}}\ \
\sup_{\{v_j\}_{j=1}^m\in H\cap H^1}
\frac{1}t
\int_0^t \sum_{j=1}^m\bigl({\mathcal L}(\tau,u_0)v_j,v_j\bigr)d\tau,
\end{equation}
where $\{v_j\}_{j=1}^m\in H\cap H^1$ is an arbitrary
orthonormal system of dimension~$m$.

We first estimate the  $H^1$-norm of the solutions on the attractor.
Taking the scalar product of  (\ref{FNSE}) with  $Au$,
 using the identity $(B(u,u),Au)=0$
(see, for example,~\cite{CF88},\cite{T}) and
integrating by parts we obtain
\begin{equation*}
\aligned
\partial_t\|\nabla u\|^2+2\nu\|Au\|^2+2\mu\|\nabla u\|^2=
2(\nabla g,\nabla u)\le \mu\|\nabla u\|^2+\mu^{-1}\|\nabla g\|^2.
\endaligned
\end{equation*}
Dropping the $\nu$-term on the left-hand side, the  Gronwall inequality
gives
$$
\|\nabla u(t)\|^2\le\|\nabla u(0)\|^2e^{-\mu t}+\frac{1-e^{-\mu t}}{\mu^2}\|\nabla g\|^2,
$$
so
 that on the attractor
$u(t)\in\mathcal{A}$ letting $t\to\infty$ we have
a $\nu$-independent estimate
\begin{equation}\label{estH1}
\|\nabla u(t)\|^2\le \frac{\|\nabla  g\|^2}{\mu^2}.
\end{equation}

We now estimate the  $m$-trace of the operator
 $\mathcal L$ in~(\ref{defqm}).
Integrating by parts and using the identity  $(B(u(t),v_j),v_j)=0$
(see~\cite{CF88},\cite{T})
and the orthonormality of the  $v_j$'s, we obtain
\begin{equation}\label{trace}
\sum_{j=1}^m\bigl({\mathcal L}(t,u_0)v_j,v_j\bigr)=
-\nu\sum_{j=1}^m\|\nabla v_j\|^2-\mu m-
\sum_{j=1}^mb(v_j,u(t),v_j).
\end{equation}

For the last term we use the point-wise inequality
\begin{equation}\label{root2}
\big|\sum_{k,i=1}^2
v^k\partial_{k}u^iv^i\big|\le
2^{-1/2}|\nabla u||v|^2
\end{equation}
which holds for any   $v=(v^1,v^2)$ and any $2\times2$
 matrix $\nabla u=\left(\partial_{i}u^j\right)_{i,j=1}^2$
 such that $\partial_{1}u^1+\partial_{2}u^2=0$, see \cite[Lemma~4.1]{Ch-I}.
Using \eqref{root2} and  the Lieb--Thirring inequality~\eqref{rhov}  we obtain
\begingroup
\allowdisplaybreaks
\begin{align*}
&\sum_{j=1}^m\bigl({\mathcal L}(t,u_0)v_j,v_j\bigr)\le
-\nu\sum_{j=1}^m\|\nabla v_j\|^2-\mu m+\frac1{\sqrt{2}}\int|\nabla u(t,x)|\rho_v(x)dx\le\\
&-\nu\sum_{j=1}^m\|\nabla v_j\|^2-\mu m+\frac1{\sqrt{2}}\|\nabla u(t)\|\|\rho_v\|\le\\
&-\nu\sum_{j=1}^m\|\nabla v_j\|^2-\mu m+\frac1{\sqrt{2}}\|\nabla u(t)\|
\left(c_{\mathrm{LT}}\sum_{j=1}^m\|\nabla v_j\|^2\right)^{1/2}\le\\
&-\nu\sum_{j=1}^m\|\nabla v_j\|^2-\mu m+\frac{c_{\mathrm{LT}}}{8\nu}\|\nabla u(t)\|^2+
\nu\sum_{j=1}^m\|\nabla v_j\|^2=\\
&-\mu m+\frac{c_{\mathrm{LT}}}{8\nu}\|\nabla u(t)\|^2.
\end{align*}
\endgroup
Now we see from \eqref{defqm} and \eqref{estH1} that
\begin{equation}\label{qm1}
q(m)\le-\mu m+\frac{c_{\mathrm{LT}}\|\nabla g\|^2}{8\nu\mu^2}.
\end{equation}

We can proceed in a somewhat different way observing that
$$
m=\int\rho_v(x)dx\le\|\rho_v\|L\qquad\Rightarrow\qquad
\sum_{j=1}^m\|\nabla v_j\|^2\ge\frac1{c_{\mathrm{LT}}}\|\rho_v\|^2\ge
\frac1{c_{\mathrm{LT}}}\frac{m^2}{L^2}\,.
$$
Then we argue as before  but in the last but one
line  single out the term $\nu /2\sum_{j=1}^m\|\nabla v_j\|^2$
to absorb it in the half of the first term. We obtain
\begin{equation}\label{qm2}
q(m)\le-\frac{\nu m^2}{2c_{\mathrm{LT}}L^2}+\frac{c_{\mathrm{LT}}\|\nabla g\|^2}{4\nu\mu^2}.
\end{equation}

Now if for an $m^*$ we have $q(m^*)<0$, then both the Hausdorff
dimension \cite{C-F85,T}, and the fractal dimension \cite{CI01,Ch-I} of the
attractor $\mathcal{A}$ satisfy
$$
\dim_H\mathcal{A}\le\dim_F\mathcal{A}<m^*.
$$
Therefore  estimates \eqref{qm1} and \eqref{qm2} along with \eqref{rhov} show that
we have proved the following theorem.
\begin{theorem}\label{T:estforsquare}
The fractal dimension of the attractor for the damped-driven
Navier-Stokes system~\eqref{DDNSE} satisfy the estimate
\begin{equation}\label{estforsquare}
\dim_F\mathcal{A}\le\min\left(\frac{c_\mathrm{LT}}8\frac{\|\nabla g\|^2}{\nu\mu^3},
\frac{c_\mathrm{LT}}{\sqrt{2}}\frac{\|\nabla g\|L}{\nu\mu} \right)\le
\min\left(\frac{3}{16}\frac{\|\nabla g\|^2}{\nu\mu^3},
\frac{3}{2\sqrt{2}}\frac{\|\nabla g\|L}{\nu\mu} \right).
\end{equation}
\end{theorem}

Let us now  consider the system \eqref{DDNSE} on the large elongated torus
$\mathbb{T}^2_\alpha=(0,L/\alpha)\times(0,L)$, where $\alpha\le1$.
As before we assume that both  scalar and vector functions have
mean value zero over $\mathbb{T}^2_\alpha$, and we decompose
the phase space
$ \dot L_2(\mathbb{T}^2_\alpha)= L_2(\mathbb{T}^2_\alpha)\cap \int_{\mathbb{T}^2_\alpha} u(x)dx=0$
into the orthogonal sum
\begin{equation}\label{Ldotsum}
 \dot L_2(\mathbb{T}^2_\alpha)=\mathrm{P}\dot L_2(\mathbb{T}^2_\alpha)\oplus
 \mathrm{Q}\dot L_2(\mathbb{T}^2_\alpha),
\end{equation}
where the orthogonal projection $\mathrm{P}$ is as in \eqref{projP1} in the 2d case,
and the projection $\mathrm{Q}$
$$
(\mathrm{Q}\psi)(x_1)=\frac1{L}\int_0^{L}
\psi(x_1,t)dt
$$
maps $\dot L_2(\mathbb{T}^2_\alpha)$ onto $\dot L_2(0,L/\alpha)$.

On the elongated torus we have the following Lieb--Thirring inequalities
for orthonormal families.
\begin{proposition}\label{T:alpha}
If
$\{v_j\}_{j=1}^m\in\mathrm{P}\dot L_2(\mathrm{T}^2_\alpha)$ is an  orthonormal
family in $L_2(\mathrm{T}^2_\alpha)$
and  $\div v_j=0$,   then
$
\rho_{\mathrm{P} v}(x):=\sum_{j=1}^m| v_j(x)|^2
$
satisfies
\begin{equation}\label{LTonP}
\int_{\mathrm{T}^2_\alpha}\rho_{\mathrm{P} v}(x)^2dx\le
c_{\mathrm{P}}\sum_{j=1}^m\|\nabla  v_j\|^2,\quad c_{\mathrm{P}}\le\frac\pi 6\,.
\end{equation}

Accordingly, if
$\{w_j\}_{j=1}^m\in\mathrm{Q}\dot L_2(\mathrm{T}^2_\alpha)$ is an
orthonormal family of  vector functions
 and  $\div w_j=0$,  then
$ \rho_{\mathrm{Q} w}(x):=\sum_{j=1}^m| w_j(x)|^2 $
satisfies
\begin{equation}\label{LTonQ}
\int_{\mathrm{T}^2_\alpha}\rho_{\mathrm{Q} w}(x)^2dx\le
\frac{c_{\mathrm{Q}}}L\sum_{j=1}^m\|\nabla  w_j\|,\quad c_{\mathrm{Q}}\le 6.
\end{equation}
\end{proposition}
\begin{proof}
The proof of \eqref{LTonP} is the same as in
Proposition~\ref{T:square}.

Turning to \eqref{LTonQ} we observe that
 $w_j$ depend only on $x_1$ and the family
 $\tilde w_j=\sqrt{L}w_j$
is orthonormal in $L_2(0,L/\alpha)$ with mean value zero. Then we use the one-dimensional
 Lieb--Thirring inequality for the operator of order one on $\dot L_2(0,L/\alpha)$
\cite{ILMS}
 $$
 \int_0^{L/\alpha}\left(\sum_{j=1}^m\tilde\psi(x)^2\right)^2dx\le
 6\sum_{j=1}^m\|\tilde\psi_j^{(1/2)}\|_{L_2(0,L/\alpha)}^2
 \le6\sum_{j=1}^m\|\tilde\psi_j^{'}\|_{L_2(0,L/\alpha)},
 $$
where the second inequality  follows from the interpolation inequality
$$
\|\psi^{(1/2)}\|^2\le\|\psi\|\|\psi'\|
$$
and orthonormality.

The conditions $w=\mathrm{Q}w=w(x_1)$ and $\div w=0$ imply that
$(w^1)'_{x_1}=0$ and hence $w^1\equiv0$, therefore the inequality for vector functions is
essentially a scalar inequality.
Returning to $\psi_j$ and $w_j$ and to integrals over $\mathrm{T}^2_\alpha$
we obtain \eqref{LTonQ}.

Finally, we observe that all the inequalities in the theorem
also hold for suborthonormal families $\{v_j\}_{j=1}^m$, that is,
satisfying
$$
\sum_{i,j=1}^m\xi_i\xi_j(v_i,v_j)\le\sum_{i=1}^m\xi_i^2
\quad\text{for every}\ \xi\in\mathbb{R}^m,
$$
and the corresponding constants do not increase~\cite{I-MS2005}.
\end{proof}

We are now prepared to state the main result of this section in
which we estimate  the  fractal dimension of the  attractor
$\mathcal{A}$ of system~\eqref{DDNSE} on the torus
$\mathbb{T}^2_\alpha$ paying special attention to the dependence of
the estimates on $\alpha\to0^+$ and $\nu\to0^+$.

\begin{theorem}\label{T:DNSEalpha}
The damped Navier--Stokes system \eqref{DDNSE} on the elongated
torus $\mathrm{T}^2_\alpha$ has the attractor $\mathcal{A}$ and
\begin{equation}\label{estdimalpha}
\dim_F\mathcal{A}\le\left(\frac{c_\mathrm{P}}2+\sqrt{c_\mathrm{P}c_\mathrm{Q}}\right)\frac{\|\nabla g\|^2}{\nu\mu^2}
\le \left(\frac\pi{12}+\sqrt{\pi}\right)\frac{\|\nabla g\|^2}{\nu\mu^3}
\end{equation}
for all sufficiently small $\nu\le 8\mu L^2$.
\end{theorem}
\begin{proof} As before we have \eqref{trace}
\begin{equation}\label{tracealpha}
\sum_{j=1}^m\bigl({\mathcal L}(t,u_0)v_j,v_j\bigr)=
-\nu\sum_{j=1}^m\|\nabla v_j\|^2-\mu m-
\sum_{j=1}^mb(v_j,u(t),v_j),
\end{equation}
and the main task is to estimate
the last term there. The main idea~\cite{Z} is decompose the solution
$u$ and the $v_j$'s as follows
$$
u=\mathrm{P}u+\mathrm{Q}u,\qquad v_j=\mathrm{P}v_j+\mathrm{Q}v_j
$$
and use $\alpha$-independent Lieb--Thirring inequalities in Theorem~\ref{T:alpha}.
We note that since the $v_j$'s are orthonormal, both
$\mathrm{P}v_j$ and $\mathrm{Q}v_j$ are suborthonormal.

Since $\partial_1\mathrm{Q}=\mathrm{Q}\partial_1$ and $\partial_2\mathrm{Q}=\mathrm{Q}\partial_2=0$,
it follows that  $\div \mathrm{P}w=\div \mathrm{Q}w=0$ if $\div w=0$.
Since  $\mathrm{Q} u$ and $\mathrm{Q} v_j$ depend only on $x_1$,
$\mathrm{Q} u^1=\mathrm{Q} v_j^1=0$ and
$\int_0^L\mathrm{P}u(x_1,x_2)dx_2=0$, it follows that
$$
b(\mathrm{Q} v_j,u,\mathrm{Q} v_j)=0,\quad b(\mathrm{Q} v_j,\mathrm{Q} u,\mathrm{P} v_j)=0,
\quad b(\mathrm{P} v_j,\mathrm{Q} u,\mathrm{Q} v_j)=0.
$$
For example,
$$
b(\mathrm{Q} v,u,\mathrm{Q} v)=\int_0^{L/\alpha}(\mathrm{Q}v^2(x_1))^2dx_1
\int_0^L\partial_2u^2(x_1,x_2)dx_2=0
$$
by periodicity.
Therefore
$$
b(v_j,u,v_j)=b(\mathrm{P} v_j,u,\mathrm{P} v_j)+b(\mathrm{Q} v_j,\mathrm{P} u,\mathrm{P} v_j)+
b(\mathrm{P} v_j,\mathrm{P} u,\mathrm{Q} v_j).
$$
Hence, in view of \eqref{root2}
\begin{equation}\label{sumb}
\sum_{j=1}^mb(v_j,u,v_j)\le\frac1{\sqrt{2}}\|\nabla u\|\|\rho_{\mathrm{P} v}\|+ \sqrt{2}
\|\nabla\mathrm{P} u\|\|\rho_{\mathrm{Q} v}\|^{1/2}\|\rho_{\mathrm{P}v}\|^{1/2},
\end{equation}
For the first term we use \eqref{LTonP} and single out $\nu/2\sum_{j=1}^m\|\nabla \mathrm{P}v_j\|^2$:
$$
\aligned
\frac1{\sqrt{2}}\|\nabla u\|\|\rho_{\mathrm{P} v}\|
\le\frac{c_{\mathrm{P}}}{4\nu}\|\nabla u\|^2+\frac\nu 2\sum_{j=1}^m\|\nabla \mathrm{P}v_j\|^2.
\endaligned
$$
For the second term we use \eqref{LTonP}, estimate
\eqref{LTonQ} in the form
$$
\|\rho_{\mathrm{Q} v}\|^2\le
\frac{c_{\mathrm{Q}}}L\sqrt{m}\left(\sum_{j=1}^m\|\nabla  w_j\|^2\right)^{1/2},
$$
and single out the terms $\nu/2\sum_{j=1}^m\|\nabla \mathrm{P}v_j\|^2$
and $\nu\sum_{j=1}^m\|\nabla \mathrm{Q}v_j\|^2$:
\begingroup
\allowdisplaybreaks
$$
\aligned
&\sqrt{2}\|\nabla\mathrm{P} u\|\|\rho_{\mathrm{Q} v}\|^{1/2}\|\rho_{\mathrm{P}v}\|^{1/2}
\le\\
&\sqrt{2}\|\nabla u\|(c_{\mathrm{P}}c_{\mathrm{Q}})^{1/4}\left(\sum_{j=1}^m\|\nabla \mathrm{P}v_j\|^2 \right)^{1/4}
\left(\frac m{L^2}\sum_{j=1}^m\|\nabla \mathrm{Q}v_j\|^2 \right)^{1/8}\le\\
&\frac{\|\nabla u\|^2\sqrt{c_{\mathrm{P}}c_{\mathrm{Q}}}}{2\nu}+
\nu\left(\sum_{j=1}^m\|\nabla \mathrm{P}v_j\|^2 \right)^{1/2}
\left(\frac m{L^2}\sum_{j=1}^m\|\nabla \mathrm{Q}v_j\|^2 \right)^{1/4}\le\\
&\frac{\|\nabla u\|^2\sqrt{c_{\mathrm{P}}c_{\mathrm{Q}}}}{2\nu}+
\frac\nu 2\sum_{j=1}^m\|\nabla \mathrm{P}v_j\|^2+
\nu\sum_{j=1}^m\|\nabla \mathrm{Q}v_j\|^2
 + \frac {\nu m}{16L^2}\,.
\endaligned
$$
\endgroup
Since $\|\nabla v\|^2=\|\mathrm{P}\nabla v\|^2+\|\mathrm{Q}\nabla v\|^2$,
we obtain
$$
\sum_{j=1}^mb(v_j,u(t),v_j)\le\frac1{2\nu}\left(\frac{c_\mathrm{P}}2+
\sqrt{c_\mathrm{P}c_\mathrm{Q}}\right)\|\nabla u(t)\|^2+
\nu\sum_{j=1}^m\|\nabla v_j\|^2+ \frac {\nu m}{16L^2}\,.
$$
Substituting this into \eqref{tracealpha} and using \eqref{estH1}
we finally obtain the estimate for the sum of $m$ global Lyapunov exponents:
$$
q(m)\le\frac1{2}\left(\frac{c_\mathrm{P}}2+
\sqrt{c_\mathrm{P}c_\mathrm{Q}}\right)\frac{\|\nabla g\|^2}{\nu\mu^2}-
\mu m\left(1-\frac {\nu }{16\mu L^2}\right).
$$
If $\nu\le 8\mu L^2$, then for
$$
m^*=\left(\frac{c_\mathrm{P}}2+\sqrt{c_\mathrm{P}c_\mathrm{Q}}\right)\frac{\|\nabla g\|^2}{\nu\mu^3}
$$
$q(m^*)\le0$, and hence
$$
\dim_F\mathcal{A}\le m^*.
$$
The proof is complete.
\end{proof}

We finally point out that estimate \eqref{estforsquare} for the square torus is
sharp as $\nu\to0^+$ and estimate \eqref{estdimalpha} is sharp as both
$\nu\to0^+$ and $\alpha\to0^+$. The lower bounds for the dimension of the attractor
are based on the characterization of the  attractor as the section at any given time
of the set of all complete trajectories bounded for $t\in \mathbb{R}$. Therefore
all stationary solutions and their unstable manifolds  lie on the attractor.
Such an unstable  stationary solution for the 2d periodic
 Navier--Stokes was first constructed in~\cite{M-S}
and is called the Kolmogorov flow.
The construction was generalized in \cite{Liu} to prove that the estimate
for the dimension obtained in \cite{C-F-T}  is logarithmically sharp.
It also applies to our periodic damped Navier--Stokes system \cite{I-M-T,I-T}.
In particular, it was shown in \cite{I-T} that
the right-hand side
\begin{equation}\label{Kolmf}
g=g_s=\begin{cases}g_1=c_1\nu^2 s^3\sin sx_2,\\
g_2=0,\end{cases}
\end{equation}
where $c_1$ is an absolute constant  and  $\mathbb{T}^2_\alpha=(0,2\pi/\alpha)\times(0,2\pi)$,
produces the stationary solution with unstable manifold
of dimension
$$
d=c_2\frac{s^2}\alpha,
$$
where $s\gg1$. Setting $s:=\sqrt{\frac\mu\nu}$ we find that
$\dim\mathcal{A}\ge d=c_2\frac\mu{\alpha\nu}$. Since $g$ is independent of $x_1$,
it follows that
$$
\|\nabla g_s\|^2=c_3\frac{\nu^4s^8}\alpha=c_3\frac{\mu^4}\alpha,
$$
so that for $g=g_s$ the dimensionless number $\frac{\|\nabla g\|^2}{\nu\mu^3}$ becomes
$$
\frac{\|\nabla g\|^2}{\nu\mu^3}=c_3\frac\mu{\alpha\nu}\,
$$
and the upper bound in \eqref{estdimalpha} for $\mathcal{A}=\mathcal{A}_s$ is supplemented with a sharp
lower bound
$$
\dim \mathcal{A}_s\ge\frac{c_2}{c_3}\frac{\|\nabla g\|^2}{\nu\mu^3},
$$
where $c_2$ and $c_3$ are some absolute constants (which can be calculated).

\subsection*{Acknowledgements}\label{SS:Acknow}
A.I.  acknowledges financial support
from the Russian Science Foundation (grant no. 14-21-00025).

\end{document}